\documentclass[a4paper,11pt]{amsart}

\usepackage[ngerman, english]{babel}
\usepackage[utf8]{inputenc}
\usepackage{pdfsync}
\usepackage{verbatim}
\usepackage[onehalfspacing]{setspace}
\usepackage{amsmath}
\usepackage{amsthm}
\usepackage{amssymb}
\usepackage{amsfonts}
\usepackage{mathtools}

\usepackage{paralist}
\theoremstyle{theorem}
\newtheorem{satz}{Theorem}[section]

\newtheorem{lemma}[satz]{Lemma}
\newtheorem{kor}[satz]{Corollary}
\newtheorem{thmx}{Theorem}

\newtheorem{prop}[satz]{Proposition}

\theoremstyle{definition}
\newtheorem{defi}[satz]{Definition}
\newtheorem{bem}[satz]{Remark}

{\begin{proof}[Beweis]}
	{\end{proof}}
\newtheorem{ex}[satz]{Example}

\newcommand{\dad}[1]{\mathrm{dad}(#1)}   

\newcommand{\RR}{\mathbb R}

\newcommand{\NN}{\mathbb N}
\newcommand{\ZZ}{\mathbb Z}

\newcommand{\id}{\text{id}}

\usepackage{tikz}
\usetikzlibrary{matrix,arrows}
\usepackage[pagebackref]{hyperref}
\allowdisplaybreaks

\title[On dynamic asymptotic dimension]{On the dynamic asymptotic dimension of \'etale groupoids}
\author[C.\ B\"onicke]{Christian B\"onicke}
\address{School of Mathematics, Statistics and Physics, Newcastle University, Newcastle upon Tyne NE1 7RU, United Kingdom}
\email{christian.bonicke@ncl.ac.uk}
\date{\today}

\subjclass[2010]{}

\begin{document}
\begin{abstract}
We investigate the dynamic asymptotic dimension for \'etale groupoids introduced by Guentner, Willett and Yu. In particular, we establish several permanence properties, including estimates for products and unions of groupoids. We also establish invariance of the dynamic asymptotic dimension under Morita equivalence.
In the second part of the article, we consider a canonical coarse structure on an \'etale groupoid, and compare the asymptotic dimension of the resulting coarse space with the dynamic asymptotic dimension of the underlying groupoid.
\end{abstract}
\maketitle

\section{Introduction}
Dynamic asymptotic dimension is a dimension theory for dynamical systems introduced by Guentner, Willett and Yu in \cite{GWY17} in the general framework of \'etale groupoids. 
Since its inception, the concept has found numerous applications: it provides upper bounds on the nuclear dimension of the resulting groupoid $\mathrm{C}^*$-algebras \cite{GWY17}, even in the presence of a twist \cite{CDGaHV22}. It is also closely related to the diagonal dimension of sub-$\mathrm{C}^*$-algebras recently introduced in \cite{LLW23}.
In \cite{BDGW23} it was shown that the groupoid homology groups of a totally disconnected \'etale groupoid vanish in all degrees exceeding the dynamic asymptotic dimension of the groupoid. 
Estimates on the dynamic asymptotic dimension (denoted henceforth by $\mathrm{dad}(\cdot)$) are known for many concrete classes of examples (see e.g. \cite{ALSS21,CJMSTD20,DS18,GWY17}). 
However, the basic theory is still not very well developed in the literature. 
The aim of this article is to remedy this situation and provide a set of general results to compute the dynamic asymptotic dimension of \'etale groupoids.

\bigskip

To this end we prove several permanence properties of dynamic asymptotic dimension, which are summarised in the theorem below.
\begin{thmx} Let $G$ and $H$ be \'etale Hausdorff groupoids.
    \begin{enumerate}
     \item If $G$ and $H$ are Morita equivalent, then $\dad{G}=\dad{H}$,
\item $\mathrm{dad}(G)= \max_{U\in \mathcal{U}} \mathrm{dad}(G|_U)$ for any finite open cover $\mathcal{U}$ of $G^0$,
    \item $\dad{G\times H}\leq \dad{G}+\dad{H}$.   
\end{enumerate}
\end{thmx}

Let us remark that it is straightforward to obtain multiplicative estimates for the dynamic asymptotic dimension of a product of groupoids. The main difficulty in proving (3) lies in reducing this crude estimate and obtain an additive estimate.

Results of this kind have a long history: let us mention the classical result \cite{E95} in topological dimension theory, the results of Dranishnikov-Bell for Gromov's asymptotic dimension \cite{BD06} and the results in \cite{BDLM08} for Assuad-Nagata dimension. Our proof is more closely modelled on the beautiful approach presented in \cite{BDLM08}, who in turn attribute some of the underlying ideas to Kolmogorov and Ostrand.  We should mention that Pilgrim in \cite{pilgrim2022,pilgrim2023} independently proved similar results for the transformation group case. However, the most natural context for dynamic asymptotic dimension is the world of groupoids. 

\bigskip

In the second part of this article, we investigate the relation between the dynamic asymptotic dimension of a groupoid $G$ and its asymptotic dimension, when equipped with a canonical coarse structure denoted by $\mathcal{E}_G$.
It was recently shown in \cite{CJMSTD20} that free actions $\Gamma\curvearrowright X$ of discrete groups on zero-dimensional second countable compact Hausdorff spaces satisfy $\mathrm{dad}(\Gamma\ltimes X)\in \{\mathrm{asdim}(\Gamma),\infty\}$. In fact it is believed, that such a result should be true beyond the zero-dimensional setting. Evidence towards this can be found in the results in \cite{ALSS21,pilgrim2023}. The most natural context for the notion of dynamic asymptotic dimension however is the world of \'etale groupoids, and in this generality nothing seems to be known so far. Hence the second part of this article is concerned with initiating the study of this question more generally. Our main contribution is the following result.

\begin{thmx}\label{Thm:dad=sup asdim}
Let $G$ be a $\sigma$-compact principal \'etale groupoid with compact and totally disconnected unit space, and let $\mathcal{E}_G$ be the canonical coarse structure on $G$. Then $$\mathrm{asdim}(G,\mathcal{E}_G)\leq \mathrm{dad}(G).$$
If we further assume $\mathrm{dad}(G)<\infty$ then
$$\dad{G}=\mathrm{asdim}(G,\mathcal{E}_G).$$
\end{thmx}

\bigskip

\noindent {\bf Acknowledgement:} I am grateful to Anna Duwenig and Rufus Willett for helpful comments on an earlier version of this paper.

\section{Permanence properties of dynamic asymptotic dimension}
Given a groupoid $G$, we will denote its set of units by $G^0$ and let $r,s\colon G\rightarrow G^0$ denote the \emph{range} and \emph{source} map of $G$, respectively. Throughout this text we will only deal with locally compact Hausdorff groupoids that are \'etale, meaning that the range and source maps are local homeomorphisms.
Note, that open subgroupoids of \'etale groupoids are automatically \'etale again.
Given a groupoid $G$ and a subset $A\subseteq G^0$, the \emph{restriction} of $G$ to $A$ is the subgroupoid $G|_A\coloneqq\{g\in G\mid s(g),r(g)\in A\}$ of $G$.
Given an \'etale groupoid $G$ and a subset $K\subseteq G$, we will write $\langle K\rangle$ for the subgroupoid of $G$ generated by $K$. If $K$ is open, then this subgroupoid is automatically open.

With this in mind we can recall the definition of dynamic asymptotic dimension given in \cite[Definition~5.1]{GWY17}.

\begin{defi}
    Let $G$ be an \'etale Hausdorff groupoid and $d\in \mathbb N$. We say that $G$ has \emph{dynamic asymptotic dimension} at most $d$, if for every open relatively compact subset $K\subseteq G$, there exists a cover of $s(K)\cup r(K)$ by $d+1$ open sets $U_0,\ldots, U_d$ such that for each $0\leq i\leq d$, the groupoid
    $$\langle K\cap G|_{U_i}\rangle$$
    is relatively compact in $G$.
\end{defi}

We will first list some elementary results. To make sense of the statement of the following lemma, recall that given a groupoid $G$ with a non-compact unit space, we can form the Alexandrov groupoid by considering $G^+\coloneqq G\cup (G^0)^+$, that is we take the Alexandrov or one-point compactification of the unit space adding no further arrows (confer \cite[Definition~7.7]{KKLRU21} for details). The following Lemma below can be proven in an elementary way just from the definition above (see \cite{CDGaHV22} for details).
\begin{lemma}\label{lem:basicpermanence}
    Let $G$ be an \'etale Hausdorff groupoid. Then the following hold:
    \begin{enumerate}
        \item $\mathrm{dad}(H)\leq \mathrm{dad}(G)$ for all closed \'etale subgroupoids $H\subseteq G$,
        \item $\mathrm{dad}(G|_U)\leq \mathrm{dad}(G)$ for all open subsets $U\subseteq G^0$, and
        \item $\mathrm{dad}(G^+)=\mathrm{dad}(G).$
    \end{enumerate}
\end{lemma}

Here is another useful permanence property that can be proven directly.
\begin{prop}\label{prop:limit}
Let $G$ be an \'etale groupoid such that $G=\bigcup_{n\in\NN} G_n$ for a nested sequence of open subgroupoids $G_n\subseteq G$ satisfying $\overline{G_n}\subseteq G_{n+1}$. Then $\mathrm{dad}(G)\leq \liminf{\mathrm{dad}(G_n)}$.
\end{prop}
\begin{proof}
We may assume that $d\coloneqq\liminf{\mathrm{dad}(G_n)}<\infty$ since otherwise there is nothing to show.
Let $K\subseteq G$ be an open, relatively compact subset. We may assume without loss of generality that $s(K)\cup r(K)\subseteq K$. Since all the $G_n$ are open, compactness of $\overline{K}$ provides an $N'\in\NN$ such that $K\subseteq \overline{K}\subseteq G_{N'}$. By definition of the $\liminf$ there exists $N\geq N'$ such that $\mathrm{dad}(G_{N+1})\leq d$. Hence there exist open subsets $U_0,\ldots ,U_d\subseteq G^0_{N+1}$ covering $s(K)\cup r(K)$ such that the closure of the subgroupoid $H_i\coloneqq \langle K\cap G|_{U_i}\rangle$ in $G_{N+1}$ is compact. Note that $H_i\subseteq G_N$ and the closure of $H_i$ in $G$ coincides with the closure of $H_i$ in $G_{N+1}$ since $\overline{G_N}\subseteq G_{N+1}$. Hence $\mathrm{dad}(G)\leq d$ as desired.
\end{proof}

\subsection{Morita invariance}
Besides the obvious notion of isomorphism for \'etale groupoids, there are various notions of (Morita) equivalence in the literature and most of them are equivalent to each other.
Let us recall the formulation that will be most useful for our purposes.
Given an \'etale groupoid $G$ and a surjective local homeomorphism $\psi\colon X\to G^0$ from another locally compact Hausdorff space $X$ onto $G^0$, we can define the \emph{ampliation}, or \emph{blow-up} of $G$ with respect to $\psi$ as
$$
    G^\psi=\{(x,g,y)\in X\times G\times X\mid \psi(x)=r(g), \psi(y)=s(g)\}.
$$
It is routine to check that $G^\psi$ with the multiplication $(x,g,y)(y,h,z)=(x,gh,z)$ and inverse map $(x,g,y)^{-1}=(y,g,x)$ is an \'etale groupoid in its own right when equipped with the relative topology from $X\times G\times X$.

We will say that two \'etale groupoids $G$ and $H$ are \emph{equivalent} if they admit isomorphic blow-ups, i.e. if there exist locally compact Hausdorff spaces $X$ and $Y$ together with surjective local homeomorphisms $\psi:X\rightarrow G^0$ and $\phi:Y\rightarrow H^0$ such that $G^\psi\cong H^\varphi$.

We refer the reader to \cite[Section~3]{FKPS19} for a detailed overview of other notions of (Morita) equivalence and in particular \cite[Proposition~3.10]{FKPS19}, where it is proved that many of the most common notions coincide.

To prove that the dynamic asymptotic dimension is invariant under Morita equivalence, we first need another permanence property that is interesting in its own right. Recall, that a continuous homomorphism $\pi:G\rightarrow H$ between two \'etale groupoids is called \emph{locally proper}, if its restriction to $G|_C$ is proper for every compact subset $C\subseteq G^0$.
Examples include the inclusion map of a closed \'etale subgroupoid $H\hookrightarrow G$, the inclusion map $G|_U\hookrightarrow G$ for an open subset $U\subseteq G^0$, or the projection map $G\ltimes X\rightarrow G$, where $G\ltimes X$ is the transformation groupoid associated with a continuous action of an \'etale groupoid $G$ on a locally compact Hausdorff space $X$.
\begin{prop}\label{prop:locprop}
	Let $G$ and $H$ be \'etale groupoids and $\pi:G\rightarrow H$ a continuous and locally proper homomorphism. Then $\mathrm{dad}(G)\leq \mathrm{dad}(H)$.
\end{prop}
\begin{proof}
	We may assume, that $\mathrm{dad}(H)=d<\infty$, since otherwise there is nothing to prove. Let $K\subseteq G$ be an open relatively compact subset. Then $\pi(K)$ is a relatively compact subset of $H$. Let $C$ be an open relatively compact subset of $H$ with $\pi(K)\subseteq C$.
	By assumption we may find open subsets $U_0,\ldots, U_d$ of $H^0$ which cover $s(C)\cup r(C)$, such that for each $0\leq i\leq d$ the the subgroupoid $H_i\coloneqq \langle C\cap H|_{U_i}\rangle$ of $H$ is relatively compact. If we let $V_i\coloneqq\pi^{-1}(U_i)\cap G^0$, then $V_i$ is an open subset of $G^0$ by continuity of $\pi$. Since $\pi$ is a homomorphism, we clearly have $s(K)\cup r(K)\subseteq \bigcup_{i=0}^d V_i$. We claim that
	$$\langle K\cap G|_{V_i}\rangle$$ is a relatively compact open subgroupoid of $G$.
	Note that $\langle K\cap G|_{V_i}\rangle\subseteq \pi^{-1}(\overline{H_i})\cap G_{\mid s(\overline{K})\cup r(\overline{K})}$. Since $\pi$ is locally proper, the latter is a compact subgroupoid of $G$. This completes the proof.
\end{proof}

Note that an application of this result to the examples mentioned above gives an alternative way to prove items (1) and (2) in Lemma \ref{lem:basicpermanence} above.
We are now ready to proceed with the proof of the main result of this section:

\begin{prop}\label{Prop:MoritaInvariance}
    The dynamic asymptotic dimension is invariant under (Morita) equivalence.
\end{prop}
\begin{proof}
    Let $\psi\colon X\to G^0$ be a surjective local homeomorphism and let
    $$
    G^\psi=\{(x,g,y)\in X\times G\times X\mid \psi(x)=r(g), \psi(y)=s(g)\}
    $$
    be the ampliation of $G$ with respect to $\psi$.
    We are going to show $\dad{G^\psi}=\dad{G}$.
    
    It is routine to check that the canonical projection $\pi\colon G^\psi\to G$ is locally proper. Hence Proposition \ref{prop:locprop} immediately yields the inequality $\dad{G^\psi}\leq \dad{G}$.
    
    For the reverse inequality suppose $d\coloneqq\dad{G^\psi}<\infty$ and let $K\subseteq G$ be an open relatively compact subset. 
    Using the assumption that $\psi$ is a local homeomorphism, find $C\subseteq X$ open and relatively compact such that $s(K)\cup r(K)\subseteq \psi(C)$.
    Then the set $L\coloneqq (C\times K\times C)\cap G^\psi$ is open and relatively compact in $G^\psi$. Hence we can find $U_0,\ldots, U_d$ covering $s(L)\cup r(L)$ such that $\langle L \cap G^\psi |_{U_i}\rangle$ is open and relatively compact in $G^\psi$.
    Then the sets $V_i\coloneqq\psi(U_i)$ form an open cover of $s(K)\cup r(K)$.
    Indeed, given $u\in s(K)$ for example, there exists a $g\in K$ such that $s(g)=u$. Since $s(K)\cup r(K)\subseteq \psi(C)$ there exist $x,y\in C$ such that $\psi(x)=r(g)$ and $\psi(y)=s(g)=u$. So $(x,g,y)\in L$. In particular, $y\in s(L)\cup r(L)$ so $y\in U_i$ for some $i$. But then $u=s(g)=\psi(y)\in \psi(U_i)=V_i$.  
    Moreover, when $g\in K$ such that $s(g),r(g)\in V_i$, then there exist $x,y\in U_i$ such that $(x,g,y)\in U_i\times K\times U_i\subseteq L$. So, $g\in \pi(\langle L \cap G^\psi |_{U_i}\rangle)$. It follows that $\langle  K\cap G|_{V_i})\subseteq \pi(\langle L \cap G^\psi |_{U_i}\rangle)$ and hence $\langle  K\cap G|_{V_i})$ is relatively compact in $G$. This verifies the inequality $\dad{G}\leq \dad{G^\psi}$ and completes the proof.
\end{proof}
\subsection{A union theorem}

This section is dedicated to the following result:
\begin{satz}\label{Thm:union}
    Let $G$ be an \'etale groupoid and $\mathcal{U}$ a finite open cover. Then
    $$\mathrm{dad}(G)=\max_{U\in \mathcal{U}} \mathrm{dad}(G|_U).$$
\end{satz}

The main technical observation needed in the proof is isolated in the following Lemma:
\begin{lemma}\label{lem:gluing}
    Let $G$ be an \'etale groupoid and $V\subseteq G^0$ an open subset. 
    Suppose $V=V_0\cup V_1$ is the union of two open subsets $V_0,V_1\subseteq G^0$ and that $K_0\subseteq K_1\subseteq K_2\subseteq G$ are open, relatively compact, satisfying $K_i=K_i\cup K_i^{-1}\cup s(K_i)\cup r(K_i)$ such that
    $$H_0\coloneqq\langle K_0\cap G|_{V_0}\rangle \subseteq K_1 \text{ and } H_1\coloneqq\langle K_1^3\cap G|_{V_1}\rangle \subseteq K_2.$$
    Then $\langle K_0\cap G|_V\rangle\subseteq K_2^5.$
\end{lemma}
\begin{proof}
    Let $g\in \langle K_0\cap G|_V\rangle$. Then we can write $g=g_1\cdots g_m$ where $g_i\in K_0$ and $s(g_i),r(g_i)\in V$ for all $1\leq i\leq m$. 
    We first note that if $1\leq k<l\leq m$ such that $r(g_k),s(g_l)\in V_1$, then $g_k\cdots g_l\in H_1$.

    Armed with this observation we have the following cases:
    \begin{enumerate}
        \item If $r(g_k)\in V_0$ for all $1\leq k\leq m$, then $g_1\cdots g_{m-1}\in H_0$ and so $g\in H_0K_0\subseteq K_2^5$.
        \item Similarly, if $s(g_k)\in V_0$ for all $1\leq k\leq m$, then $g\in K_0 H_0\subseteq K_2^5$.
        \item  In all other cases, there must exist indices $s,t$ such that $g_s$ is the first element such that $r(g_s)\in V_1$ and $g_t$ is the last element such that $s(g_t)\in V_1$. In this case
        \begin{align*}
            g=&(g_1\cdots g_{s-2})g_{s-1}(g_s\cdots g_t)g_{t+1}(g_{t+2}\cdots g_n)\\
            &\in H_0K_0H_1K_0H_0\subseteq K_1^2 K_2 K_1^2\subseteq K_2^5.
        \end{align*}
    \end{enumerate}
\end{proof}

\begin{proof}[Proof of Theorem \ref{Thm:union}.]
    It follows from Lemma \ref{lem:basicpermanence} that $\mathrm{dad}(G|_U)\leq \mathrm{dad}(G)$ for all $U\in \mathcal{U}$. Hence we only need to verify the inequality $\mathrm{dad}(G)\leq \max_{U\in \mathcal{U}} \mathrm{dad}(G|_U)$. 
    Inductively, it will be enough to show this for a cover by two open sets $U_0$ and $U_1$. Moreover, using Morita invariance, we may pass to the groupoid $G[\mathcal{U}]$ to assume that the $U_0$ and $U_1$ are disjoint, i.e. a clopen cover. Given an open, relatively compact subset $K\subseteq G$ the set $K\cap G|_{U_0}$ is open and relatively compact in $G|_{U_0}$. So by assumption, we may find an open cover $V_{0,0},\ldots, V_{0,d}$ of $(s(K)\cup r(K))\cap U_0$ such that $\langle K\cap G|_{V_{0,i}}\rangle$ is relatively compact. Now set $K_1\coloneqq K\cup  \bigcup_{i=0}^d \langle K\cap G|_{V_{0,i}}\rangle$. Then $K_1$ is open and relatively compact in $G$. 
    Using our assumption on $G|_{U_1}$ now, we can find an open cover $V_{1,0},\ldots, V_{1,d}$ of $s(K_1)\cup r(K_1)$ such that $\langle K_1^3\cap G|_{V_{1,i}}\rangle$ is relatively compact for all $0\leq i\leq d$.
    Set $V_i\coloneqq V_{0,i}\cup V_{1,i}$. Then $V_0,\ldots, V_d$ is an open cover of $s(K)\cup r(K)$ and Lemma \ref{lem:gluing} implies that for each $0\leq i\leq d$ the groupoid $\langle K\cap G|_{V_i}\rangle$ is relatively compact.
\end{proof}

Note, that we can also combine Theorem \ref{Thm:union} with Proposition \ref{prop:limit} to obtain an estimate for infinite covers. We will leave the details to the reader.

\subsection{A product theorem}
The goal of this subsection is to establish an estimate for the dynamic asymptotic dimension of the product of two \'etale groupoids.

\begin{satz}\label{Thm:Product}
    Let $G$ and $H$ be \'etale Hausdorff groupoids. Then
    $$\dad{G\times H}\leq \dad{G}+\dad{H}.$$
\end{satz}
Note, that one can find a multiplicative estimate for $\dad{G\times H}$ in a straightforward fashion from the definitions. The proof of the additive formula in Theorem \ref{Thm:Product} however is more involved.

Let $\mathcal{O}_{c}(G)$ denote the set of open relatively compact subsets $K\subseteq G$ such that $K=K\cup K^{-1}\cup s(K)\cup r(K)$. If $G^0$ is compact we will additionally require that $G^0\subseteq K$.
\begin{defi}
Let $G$ be an \'etale groupoid with compact unit space. 
A $d$-dimensional control function for $G$ is a map $D_G:\mathcal{O}_c(G)\to \mathcal{O}_c(G)$ such that for every $K\in \mathcal{O}_c(G)$ there exists an open cover $U_0,\ldots, U_d$ of $G^0$ such that the groupoid $\langle K\cap G|_{U_i}\rangle$ is contained in $D_G(K)$ for all $0\leq i\leq d$.
\end{defi}
Note that a $d$-dimensional control function for $G$ exists if and only if $\mathrm{dad}(G)\leq d$.
The following definition and results are inspired by \cite{BDLM08}. In order to state them let us introduce the following terminology: A collection $\{U_i\mid i\in I\}$ of open subsets of topological space $X$ is called an \emph{$n$-fold open cover} if $\{i\in I\mid x\in U_i\}$ has cardinality at least $n\in \mathbb N$.

\begin{defi}
Let $k\geq d\geq 0$. A $(d,k)$-dimensional control function for $G$ is a map $D_G:\mathcal{O}_c(G)\longrightarrow \mathcal{O}_c(G)$ such that for every $K\in \mathcal{O}_c(G)$ there exists an $(k+1-d)$-fold open cover $U_0,\ldots, U_k$ of $G^0$ such that the groupoid $\langle K\cap G|_{U_i}\rangle$ is contained in $D_G(K)$ for all $0\leq i\leq k$.
\end{defi}
We note that a $(d,d)$-dimensional control function is just a $d$-dimensional control function in the sense of the previous definition.
There is also a version of the dimension control function for continuous homomorphisms:

To prove the main technical proposition in this subsection we need the following facts about $n$-fold covers:
\begin{lemma}\label{lem:ncovers}
    Let $X$ be a compact Hausdorff space.
    \begin{enumerate}
        \item A collection $\{U_0,\ldots, U_d\}$ of open subsets of $X$ is an $n$-fold cover of $X$ if and only if $\{U_i\mid i\in F\}$ is a cover of $X$ for every subset $F\subseteq \{0,\ldots,d\}$ of cardinality $d+2-n$.
        \item If $\{U_0,\ldots, U_d\}$  is an $n$-fold open cover of $X$, then there exist open subsets $V_i\subseteq \overline{V_i}\subseteq U_i$ for all $0\leq i\leq d$ such that $\{V_0,\ldots, V_d\}$ is still an $n$-fold open cover of $X$.
    \end{enumerate}
\end{lemma}
\begin{proof}
    \begin{enumerate}
        \item For the forward direction, let $\{U_0,\ldots, U_d\}$ be an $n$-fold open cover of $X$, and let $F\subseteq \{0,\ldots,d\}$ be a subset of cardinality $d+2-n$. Then $F^c$ has cardinality $d+1-(d+2-n)=n-1$. Since every $x\in X$ is contained in at least $n$ of the sets  $\{U_0,\ldots, U_d\}$, there must exist an $i\in F$ such that $x\in U_i$. Hence $\{U_i\mid i\in F\}$ is a cover of $X$.
        Conversely, assume for contradiction that there exists an $x\in X$ contained in at most $n-1$ members of the cover. Then $F=\{i\mid x\in U_i\}$ has cardinality at most $n-1$ and it follows that $F^c$ has cardinality at least $d+1-(n-1)=d+2-n$. By assumption $x$ must then be contained in $U_i$ for some $i\in F^c$ which contradicts our choice of $F$.
        \item Using the first part of this Lemma, $\{U_i\mid i\in F\}$ is an open cover for each $F\subseteq \{0,\ldots,d\}$ of cardinality $d+2-n$. Since $X$ is compact, it is normal. Hence for each fixed $F$, there exist open subsets  $V_i^F\subseteq X$ such that $\overline{V_i^F}\subseteq U_i$ and such that $\{V_i^F\mid i\in F\}$ still covers $X$. Let $V_i\coloneqq\bigcup_{F\ni i}V_i^F$. Then $V_i$ is open, $\overline{V_i}\subseteq U_i$. Moreover, for any finite subset $F$ of cardinality $d+2-n$ the collection $\{V_i\mid i\in F\}$ covers $X$ since each $V_i$ contains the set $V_i^F$ which already form a cover of $X$. Another application of item (1) concludes the proof. 
    \end{enumerate}
\end{proof}

The following proposition shows how to obtain $(d,k)$-dimensional control functions for every $k\geq d$ starting from a $(d,d)$-dimensional control function. This is the main technical ingredient needed to prove the main results.

\begin{prop}\label{prop:ostrand} Let $G$ be an  \'etale groupoid with compact unit space.
If $G$ admits a $d$-dimensional control function $D_G$, set $D_G^{(d)}\coloneqq D_G$ and inductively define functions $D_G^{(k)}$ for $k\geq d$ by $D_G^{(k+1)}(K)=KD_G^{(k)}(K^3)K$. Then $D_G^{(k)}$ is a $(d,k)$-dimensional dimension function for all $k\geq d$.
\end{prop}
\begin{proof}
We proceed by induction on $k\geq d$. Since $G$ admits a $d$-dimensional control function by assumption, the base case $k=d$ is obvious.
Suppose now that $D_G^{(k)}$ is a $(d,k)$-dimensional dimension function and let $K\in \mathcal{O}_c(G)$ be an open relatively compact subset.
Then $K^3\in \mathcal{O}_c(G)$ as well. Hence the induction hypothesis provides a $(k+1-d)$-fold open cover $U_0,\ldots, U_k$ of $G^0$ such that $\langle K^3 \cap G|_{U_i}\rangle\subseteq D_G^{(k)}(K^3)$ for all $0\leq i\leq k$.
Let $U_i'\coloneqq KU_i$ be the $K$-orbit of $U_i$. Note that $U_i\subseteq U_i'$ since $G^0\subseteq K$.

\noindent\textbf{Claim:} $\langle K\cap G|_{U'_i}\rangle\subseteq KD_G^{(k)}(K^3)K=D_G^{(k+1)}(K)$.

\begin{proof}[Proof of Claim.]
    Let $g\in \langle K \cap G|_{U_i'}\rangle$. This means that $g=g_n\cdots g_1$ for $g_1,\ldots, g_n\in K$ such that $s(g_j),r(g_j)\in U_i'=KU_i$. It follows that for each $k$, there exists a $h_j\in K$ such that $s(h_j)=s(g_j)$ and $r(h_j)\in U_i$. If we set $g_j'\coloneqq h_{j+1}g_jh_j^{-1}\in K^3$. Note that $s(g_j'),r(g_j')\in U_i$ and hence it follows that
    $$g=h_{n+1}^{-1}g_n'\cdots g_1'h_1\in K \langle K^3 \cap G|_{U_i}\rangle K\subseteq K D^{(k)}_G(K^3) K=D_G^{(k+1)}(K).$$
\end{proof}

Let $V_i\subseteq \overline{V_i}\subseteq U_i$ be such that
\begin{enumerate}
    \item $V_0,\ldots, V_k$ is still a $(k+1-d)$-fold cover, and
    \item $\overline{KV_i}\subseteq KU_i=U_i'$.
\end{enumerate}
For a $1$-cover this can be found using \cite[Lemma~7.4]{GWY17}. In the general case we can also follow the proof of this result, but use Lemma \ref{lem:ncovers} above, when shrinking covers.

The additional open set needed at stage $k+1$ will be the set
$$U_{k+1}'\coloneqq\bigcup_S \left(\bigcap_{j\in S} V_j\setminus \bigcup_{i\notin S} \overline{KV_i}\right)$$ where $S$ runs through the subsets of $\{0,\ldots, k\}$ of cardinality $k+1-d$.
It is clear that $U_{k+1}'$ is open.

We claim that $\langle K\cap G|_{U_{k+1}'}\rangle\subseteq KD_G^{(k)}(K^3)K$. Suppose that $g=g_n\cdots g_1$ with $g_l\in  K$ and $s(g_l),r(g_l)\in U_{k+1}'$ for all $1\leq l\leq n$. Then there exist subsets $S_1,\ldots, S_{n+1} \subseteq \{0,\ldots,k\}$ of cardinality $k+1-d$ such that $s(g_l)\in \bigcap_{j\in S_l} V_j\setminus \bigcup_{i\notin S_l} \overline{KV_i}$ and $r(g_n)\in \bigcap_{j\in S_{n+1}} V_j\setminus \bigcup_{i\notin S_{n+1}} \overline{KV_i}$.
Observe, that $S_1=\ldots =S_{n+1}$. Indeed, suppose for contradiction that there is some index $1\leq l\leq n$ such that $S_{l+1}\neq S_l$. Then we may assume without loss of generality that there exists an $i\in S_l\setminus S_{l+1}$. Since $i\in S_l$ we have $s(g_l)\in V_i$. But then $s(g_{l+1})=r(g_l)=g_ls(g_l)\in KV_i\subseteq \overline{KV_i}$ for $i\in S_{l+1}^c$, a contradiction. 

Since $S_{l+1}=S_l$ for all $1\leq l\leq n$, $s(g_l),r(g_l)\in V_j\subseteq U_j$ for all $j\in S_{l+1}=S_l$ and hence $g_l\in \langle K\cap G|_{U_j}\rangle$. But then $g\in \langle K\cap G|_{U_j}\rangle\subseteq D_G^{(k)}(K^3)\subseteq KD_G^{(k)}(K^3) K$.

Finally, we claim $U_0',\ldots, U_{k+1}'$ is a $((k+1)+1-d)$-fold open cover of $G^0$. We know that $V_0,\ldots, V_{k}$ is a $(k+1-d)$-fold by the induction hypothesis. Fix $x\in G^0$. If it belongs to $k+2-d$ among the sets $U_0',\ldots, U_{k}'$ we are done. So let us assume that it belongs exactly to $k+1-d$ of the sets $U_0',\ldots, U_{k}'$. To complete the proof 
we will show that $x\in U_{k+1}'$. To see this note that the assumption together with the fact that $V_i\subseteq U_i'$ implies that $S=\{i\leq k\mid x\in V_i\}$ has cardinality $k+1-d$. Moreover, $x\in \bigcap_{i\in S}V_i$ and our hypothesis implies $x\notin \bigcup_{i\in \{0,\ldots,k\}\setminus S}\overline{KV_i}$, which together exactly means that $x\in U_{k+1}'$.
\end{proof}

\begin{proof}[Proof of Theorem \ref{Thm:Product}]
We will first prove the result in the case that $G^0$ and $H^0$ are compact. 
We may assume that $\mathrm{dad}(G)$ and $\mathrm{dad}(H)$ are both finite. Set $k\coloneqq\mathrm{dad}(G)+\mathrm{dad}(H)$. By Proposition \ref{prop:ostrand} we may find a $(\mathrm{dad}(G),k)-$dimensional dimension function $D_G$ for $G$ and a $(\mathrm{dad}(H),k)$-dimensional dimension function $D_H$ for $H$. Now let $C\subseteq G\times H$ be open and relatively compact. Since increasing $C$ only makes the problem harder, we may assume that $C=K\times L$ for open, relatively compact subsets $G^0\subseteq K\subseteq G$ and $H^0\subseteq L\subseteq H$. 
Now find a $(k+1-\mathrm{dad}(G))$-fold open cover of $G^0$ such that $\langle K\cap G|_{U_i}\rangle\subseteq D_G(K)$ for all $0\leq i\leq k$ and a $(k+1-\mathrm{dad}(H))$-fold open cover $V_0,\ldots, V_k$ of $H^0$ such that $\langle L\cap H|_{U_i}\rangle\subseteq D_H(L)$.
We claim that the sets $U_0\times V_0,\ldots, U_k\times V_k$ form an open cover of $G^0\times H^0$. Indeed, let $(x,y)\in G^0\times H^0$.
By our choices of the cover $(U_i)_i$ above, the set
$\{i\mid x\in U_i\}$ has cardinality at least $\mathrm{dad}^{+1}(H)$ and similarly, the set $\{i\mid y\in V_i\}$ has cardinality at least $\mathrm{dad}^{+1}(G)$. Since both of them are subsets of $\{0,\ldots, k\}$, their intersection cannot be empty, which proves our claim.
To complete the proof note that
$$\langle C\cap (G\times H)|_{U_i\times V_i}\rangle \subseteq \langle K\cap G|_{U_i}\rangle \times \langle L\cap H|_{V_i}\rangle\subseteq D_G(K)\times D_H(L).$$

Finally, consider the case that $G^0$ and $H^0$ are merely locally compact. Note that $G^0$ and $H^0$ are open in their respective one-point compactifications and $(G^+\times H^+)|_{G^0\times H^0}=G\times H$. Hence Lemma \ref{lem:basicpermanence} allows us to compute
$$\mathrm{dad}(G\times H)\leq \mathrm{dad}(G^+\times H^+)\leq \mathrm{dad}(G^+)+\mathrm{dad}(H^+)=\mathrm{dad}(G)+\mathrm{dad}(H).$$
\end{proof}
In \cite[Proposition~2.3]{WZ10} the authors prove a multiplicative formula for the nuclear dimension of tensor products of $\mathrm{C}^*$-algebras.
Combining Theorem \ref{Thm:Product} with the main results in \cite{GWY17,CDGaHV22} yields improved estimates for the nuclear dimension of tensor products of groupoid $\mathrm{C}^*$-algebras:
\begin{kor}
    Let $(G_1,\Sigma_1)$ and $(G_2,\Sigma_2)$ be two twisted \'etale groupoids. Then
    \begin{align*}
        \dim^{+1}_{nuc}(C_r^*(G_1;\Sigma_1)&\otimes C_r^*(G_2;\Sigma_2))\leq \\
        &(\mathrm{dad}(G_1)+\mathrm{dad}(G_2)+1)(\dim(G_1^0)+\dim(G_2^0)+1).
    \end{align*}
\end{kor}
\begin{proof}
    We only need to note that $C_r^*(G_1,\Sigma_1)\otimes C_r^*(G_2,\Sigma_2)\cong C_r^*(G_1\times G_2,\Sigma_1\times \Sigma_2)$ and apply the results mentioned above.
\end{proof}

\begin{bem}
    It seems reasonable to expect that there is a general Hurewicz type result for the dynamic asymptotic dimension. To be a little more precise: If $\pi\colon G\to H$ is a continuous homomorphism between two \'etale groupoids, we can define the dynamic asymptotic dimension of $\pi$ as
    $$\mathrm{dad}(\pi)\coloneqq \sup \{\mathrm{dad}(\pi^{-1}(L))\mid L\subseteq H \text{ open subgroupoid s.t. } \mathrm{dad}(L)=0\}.$$

    We have the following examples:
    \begin{enumerate}
    \item If $\pi:G\rightarrow H$ is locally proper, then $\mathrm{dad}(\pi)=0$.
    \item If $\pi_G:G\times H\rightarrow G$ is the projection, then $\mathrm{dad}(\pi_G)=\mathrm{dad}(H)$.
    \end{enumerate}

    With these two examples in mind we conjecture that for any continuous groupoid homomorphism $\pi:G\rightarrow H$ we have
    $$\mathrm{dad}(G)\leq \mathrm{dad}(\pi)+\mathrm{dad}(H).$$
\end{bem}

\subsection{Applications to partial actions}
To illustrate our results, let us discuss a class of \'etale groupoids arising from partial actions. 
A \emph{partial action} of a discrete group $\Gamma$ on a locally compact Hausdorff space $X$ is a pair $\theta=((D_\gamma)_{\gamma\in \Gamma}, (\theta_\gamma)_{\gamma\in \Gamma})$ where $D_\gamma\subseteq X$ is an open subset for all $\gamma\in \Gamma$ and $\theta_\gamma\colon D_{\gamma^{-1}}\to D_\gamma$ is a homeomorphism such that $D_1=X$, $\theta_1=\id_X$, and such that $\theta_{\gamma\eta}$ extends $\theta_\gamma \circ \theta_\eta$ (where the latter is defined on $\theta_\eta^{-1}(D_\eta \cap D_{\gamma^{-1}})$).
Associated with such a partial action is the \'etale groupoid $$\Gamma\ltimes_\theta X\coloneqq \{(\gamma,x)\in \Gamma\times X\mid x\in D_{\gamma}\}$$
where the multiplication is defined as $(\gamma,x)(\eta,\theta_\gamma^{-1}x)=(\gamma\eta, x)$ whenever $\theta_\gamma^{-1}(x)\in D_\eta$.

\begin{kor} Let $\Gamma$ be a countable discrete group in the class of groups described in \cite[Theorem~10.7]{CJMSTD20} and let $\theta=((D_\gamma)_{\gamma\in \Gamma}, (\theta_\gamma)_{\gamma\in \Gamma})$ be a free partial action of $\Gamma$ on a zero-dimensional Hausdorff space $X$. If $D_\gamma$ is clopen for all $\gamma\in \Gamma$, then
    $\mathrm{dad}(\Gamma\ltimes_\theta X)\leq \mathrm{asdim}(\Gamma)$.
\end{kor}
\begin{proof}
    The assumptions imply that the partial action admits a Hausdorff globalisation. In other words, there exists a free global action of $\Gamma$ on some locally compact Hausdorff space $Y$ containing $X$ as an open subset and such that $\Gamma\ltimes X$ is Morita equivalent to $\Gamma\ltimes Y$. Note that $\mathrm{dim}(Y)=0$ as well by \cite[Proposition~3.2]{G21}. Hence Proposition \ref{Prop:MoritaInvariance} implies that 
$\mathrm{dad}(\Gamma\ltimes X)=\mathrm{dad}(\Gamma\ltimes Y)$ and  by \cite[Theorem~10.7]{CJMSTD20}, the latter is bounded above by $\mathrm{asdim}(\Gamma)$.
\end{proof}

Of course the collection $D_\gamma$ need not always be closed. Nevertheless, we should expect the same upper bound on the dynamic asymptotic dimension as the following examples shows:

\begin{ex}
    Let $X$ be a zero-dimensional metrisable space and $\theta:U\rightarrow V$ a homeomorphism between two open sets such that $\theta$ generates a free partial action of $\mathbb Z$. Let $$\ZZ\ltimes_\theta X\coloneqq \{(n,x)\mid x\in D_n\}$$
    be the associated transformation groupoid. 
    Using \cite[Corollary~3.6]{G21} there exist partial actions $\theta^{(k)}$ for $k\in \mathbb N$ such that each $\theta^{(k)}$ admits a Hausdorff globalisation, and such that $\mathbb Z\ltimes_\theta X$ can be written as an increasing union
    $$\mathbb Z\ltimes_\theta X=\bigcup_{k\in \mathbb N} \ZZ\ltimes_{\theta^{(k)}} X.$$ 
    Thus, for each $k\in \mathbb N$ there exists a zero-dimensional  Hausdorff space $Y_k$, and a (global) action $\mathbb Z\curvearrowright Y_k$ such that $\mathbb Z\ltimes Y_k$ is Morita equivalent to $\mathrm{Z}\ltimes_{\theta^{(k)}} X$.
    Proposition \ref{Prop:MoritaInvariance} implies $\mathrm{dad}(\ZZ\ltimes_{\theta^{(k)}}X)\leq 1$ for all $k\in \mathbb N$ and hence $\mathrm{dad}(\ZZ\ltimes_\theta X)\leq 1$ as well by Proposition \ref{prop:limit}.
\end{ex}

\section{Asymptotic dimension}
In this second part of the article we will compare the dynamic asymptotic dimension of an \'etale groupoid $G$ with the classical asymptotic dimension of $G$ with respect to a canonical coarse structure on $G$. Coarse structures on \'etale groupoids have been studied before by other authors, see for example \cite[Remark~4.15]{MW20}.

Let us first specify which coarse structure we want to consider on a $\sigma$-compact étale groupoid $G$:
Let $\mathcal{E}_G$ be the collection of subsets of $\{(g,h)\in G\times G\mid r(g)=r(h)\}$, such that $E\in\mathcal{E}_G$ if there exists an open relatively compact subset $K\subseteq G$, such that $E\subseteq \lbrace (g,h)\mid g^{-1}h\in K\rbrace\cup \Delta_G$. Then $\mathcal{E}$ is a coarse structure on $G$.
The elements of $\mathcal{E}_G$ are called \emph{controlled} sets.
Note, that the coarse structure on $G$ also induces a coarse structure $\mathcal{E}_{G^x}$ on each of the range fibres $G^x$ by intersecting each controlled set $E\in \mathcal{E}_G$ with $G^x\times G^x$.

\begin{ex}\label{Ex:trafogrpd}
Let $\Gamma\ltimes X$ be the transformation groupoid for an action of a countable discrete group $\Gamma$ on a locally compact space $X$.
Restricting the canonical coarse structure considered above to any range fibre $(\Gamma\ltimes X)^x$ and identifying it with $\Gamma$ in the canonical way, gives rise to the coarse structure on $\Gamma$ described by Roe in \cite[Example~2.13]{R03}.
\end{ex}
Let us now recall the definition of asymptotic dimension:
If $E$ is a controlled set for a coarse space $(X,\mathcal{E})$, then a family $\mathcal{U}=\lbrace U_i\rbrace_{i\in I}$ of subsets of $X$ is called \emph{$E$-separated} if $(U_i\times U_j)\cap E=\emptyset$ for all $i\neq j$, and \emph{$E$-bounded} if $U_i\times U_i\subseteq E$ for all $i\in I$.

Moreover, $X$ is said to have \emph{asymptotic dimension at most} $d$ if $d$ is the smallest number with the following property: For any controlled set $E$ there exists a controlled set $F$ and a cover $\mathcal{U}$ of $X$ which is $F$-bounded and admits a decomposition
$$\mathcal{U}=\mathcal{U}_0\sqcup\ldots\sqcup \mathcal{U}_d$$
such that each $\mathcal{U}_i$ is $E$-separated.

Since the asymptotic dimension of a subspace is at most the asymptotic dimension of the ambient space, we have the obvious estimate
\begin{equation}\label{Eq:asdim}
    \sup_{x\in G^0} \mathrm{asdim}(G^x,\mathcal{E}_{G^x})\leq \mathrm{asdim}(G,\mathcal{E}_G).
\end{equation}
In the case of a transformation groupoid $\Gamma\ltimes X$ considered in Example \ref{Ex:trafogrpd} all the range fibres canonically identify with the group $\Gamma$ itself and hence the asymptotic dimension of $(\Gamma\ltimes X,\mathcal{E}_{\Gamma\ltimes X})$ coincides with the asymptotic dimension of $\Gamma$ with respect to the canonical coarse structure.

However, the reverse of inequality (\ref{Eq:asdim}) may fail, because even if $\mathrm{asdim}(G^x)<\infty$ for all $x\in G^0$ the sets $F_x$ obtained from the definition of asymptotic dimension that are controlling the size of the members of the cover, may grow in an uncontrollable manner as $x$ varies across $G^0$.
Another example where we do get an equality is the following:

\begin{ex}
	A \textit{graphing} on an \'etale groupoid $G$ is an open relatively compact set $Q\subseteq G\setminus G^0$ with $Q=Q^{-1}$ that generates $G$ in the sense that $G=\bigcup_{n=1}^\infty Q^n$, where we adopt the convention that $Q^0=G^0$. We say that $G$ is \textit{treeable} $G$ admits a graphing such that every $g\in G\setminus G^0$ has a unique (reduced) factorisation $g=g_m\cdots g_1$ with $g_i\in Q$.
	
	Treeable groupoids are in some sense the analogues of free groups in the world of groupoids. Indeed, if $S$ denotes a free generating set for $\mathbb{F}_n$ and we are given an action $\mathbb{F}_n\curvearrowright X$ on a compact space, then $\mathbb{F}_n\ltimes X$ is a treeable groupoid with $Q=S\times X$.
	
	We are now going to show that $\mathrm{asdim}(G,\mathcal{E}_G)=1$ for any treeable groupoid $G$. Let $Q$ be a graphing as in the definition above. This graphing induces a length function $\ell:G\rightarrow [0,\infty)$ given as the length of the unique reduced factorisation of $g$ as a product of elements in $Q$. The function $\ell$ is continuous since $Q$ is open, and controlled and proper since $Q$ is relatively compact. We denote by $B_N\coloneqq \{g\in G\mid \ell(g)\leq N\}=\bigcup_{n=0}^N Q^n$.
	
	Let $K\subseteq G$ be open and relatively compact. Then there exists an $N\in\NN$ such that $K\subseteq B_N$. We consider the ``annuli"
	$$A_k:=\{ g\in G\mid kN\leq \ell(g)\leq (k+1)N\}$$
	If we only consider those annuli indexed by even (resp. odd) numbers $A_{2k}$ (resp. $A_{2k+1}$) then these families are pairwise $K$-disjoint. However, they are not yet uniformly bounded. Hence we need to further subdivide each annulus.
	For $k\geq 2$ we define an equivalence relation on $A_k$ by setting $g\sim_k h$ if $g$ and $h$ have the same past up to distance $N(k-1)$ from the origin, i.e. if $g=g_1\cdots g_m$ and $h=h_1\cdots h_l$ are the unique reduced factorisations of $G$ with elements in $Q$, then $g_i=h_i$ for all $i\leq N(k-1)$. This is clearly an equivalence relation on $A_k$ and we denote the equivalence class of $g\in A_k$ by $[g]_k$.
	Let $h\in [g]_k$. If we factorise $g=g_1\cdots g_{\ell(g)}$ then $h=g_1\cdots g_{N(k-1)}h_{N(k-1)+1}\cdots h_{\ell(h)}$ and hence $d(g,h)=\ell(g^{-1}h)=\ell(g^{-1}_{\ell(g)}\cdots g_{N(k-1)}^{-1}h_{N(k-1)+1}\cdots h_{\ell(h)})\leq \ell(g)+\ell(h)-2N(k-1)-1\leq 2N(k+1) -2N(k-1)=4N$. Hence the diameter of each equivalence class is uniformly bounded by $4N$. Since $A_0$ and $A_1$ are also bounded we can set $\mathcal{U}_0=\bigcup_{k=1}^\infty \{ [g]_{2k}\mid g\in A_{2k}\}\cup\{A_0\}$ and $\mathcal{U}_1=\bigcup_{k=1}^\infty \{ [g]_{2k+1}\mid g\in A_{2k+1}\}\cup\{A_1\}$. Then $\mathcal{U}_0\cup \mathcal{U}_1$ is a uniformly bounded cover of $G$. Moreover, each $\mathcal{U}_i$ is $N$-disjoint. As we have seen, the even and odd annuli are already $N$-disjoint, so we only have to check for $N$-disjointness within each $A_k$ for $k\geq 2$. So fix $k\geq 2$ and let $g,h\in A_k$ be such that they do not agree on the first $N(k-1)$ elements in the unique factorisation. Let $j\in \{1,\ldots ,N(k-1)\}$ be the minimal number such that $g_{j}\neq h_{j}$ in the unique factorisations of $g$ and $h$. Then we compute
	\begin{align*}
	d(g,h)&=\ell(g)+\ell(h)-2(j-1)\\
	&\geq \ell(g)+\ell(h) - 2N(k-1)\\
	&\geq 2kN-2N(k-1)=2N\geq N.
	\end{align*}
	This finishes the proof.
\end{ex}

The following Proposition gives the first half of Theorem \ref{Thm:dad=sup asdim}.
\begin{prop}
Let $G$ be an étale groupoid with compact unit space. Then
$$\mathrm{asdim}(G,\mathcal{E}_G)\leq \mathrm{dad}(G).$$
\end{prop}
\begin{proof}
	We may assume that $d\coloneqq\mathrm{dad}(G)<\infty$ since otherwise there is nothing to prove. Let $E$ be a controlled set for $G$. Since $G^0$ is compact, there exists a relatively compact open subset $G^0\subseteq K\subseteq G$ such that
	$$E\subseteq\lbrace (g,h)\mid g^{-1}h\in K\rbrace.$$ 
	Using the assumption, find open subsets $U_0,\ldots, U_d$ covering $G^0$ such that the subgroupoid $H_i\coloneqq \langle K\cap G|_{U_i}\rangle$ of $G$
	is relatively compact (and open) for every $i=0,\ldots,d$. Let $F\coloneqq\lbrace (g,h)\mid g^{-1}h\in \bigcup_{i=0}^d H_i\rbrace$.
	Then $F$ is by its very definition a controlled set for $G$. 
	Let $\sim_i$ denote the equivalence relation on $G_{U_i}$ given by
	$$g\sim_i h :\Leftrightarrow r(g)=r(h)\text{ and }g^{-1}h\in H_i.$$
	Let $\mathcal{U}_i\coloneqq\lbrace [g]_i\mid g\in G_{U_i}\rbrace$ be the collection of all equivalences classes of the relation $\sim_i$.
	
	Then each $\mathcal{U}_i$ is $E$-separated. Indeed, suppose $[g]_i\times[h]_i\cap E\neq \emptyset$. Then there exist $g_0,h_0\in G_{U_i}$ such that $g_0\sim_i g$ and $h_0\sim_i h$ such that $(g_0,h_0)\in E$. In particular, we have $g_0^{-1}h_0\in K$, which implies $g_0\sim_i h_0$ and hence $[g]_i=[h]_i$.
	
	Moreover, the collection $\mathcal{U}=\bigcup_{i=0}^d \mathcal{U}_i$ is an $F$-bounded cover for $G$:
	We have to show $[g]_i\times[g]_i\subseteq F$ for all $i=0,\ldots, d$ and $g\in G_{U_i}$.
	If $(g_1,g_2)\in [g]_i\times[g]_i$ then $g_1\sim_i g\sim_i g_2$ and hence $g_1^{-1}g_2\in H_i$ and hence $(g_1,g_2)\in F$.
\end{proof}

\begin{bem}
We remark that our assumption that $G^0$ is compact in the previous proposition cannot be relaxed. Consider for example the action of the integers $\ZZ$ on $\RR$ by translation. This action is free and proper. Using this it is not hard to show that $\mathrm{dad}(\ZZ\ltimes \RR)=0$. On the other hand, $\mathrm{asdim}(\ZZ)=1$.
\end{bem}

In what follows we want to prove the converse in the case that $G^0$ is zero-dimensional. The proof is inspired by the recent article \cite{CJMSTD20}.

We need a variant of dynamic asymptotic dimension that keeps track of the size of the subgroupoids obtained in the definition.
Given open relatively compact subsets $K, L\subseteq G$, we say that an open subgroupoid $H\subseteq G$ has $(K,L)$-dad at most $d$ if there exists a cover of $H^0\cap (s(K)\cup r(K))$ by open sets $U_0,\ldots, U_d$ such that $\langle  K\cap G|_{U_i}\rangle\subseteq L$ for all $0\leq i\leq d$.

Similarly, a coarse space $(X,\mathcal{E})$ has $(E,F)-\mathrm{asdim}$ at most $d$ if there exists a cover $\mathcal{U}$ of $X$ which is $F$-bounded and admits a decomposition $\mathcal{U}=\mathcal{U}_0\sqcup \ldots \sqcup \mathcal{U}_d$ such that each $\mathcal{U}_i$ is $E$-separated.

\begin{lemma}
    Let $G$ be an \'etale groupoid and $V\subseteq G^0$ an open subset such that $V=V_0\cup\cdots\cup V_n$ for some open subsets $V_i\subseteq G^0$. 
    Suppose further that $K_0\subseteq K_1\subseteq \ldots\subseteq K_{n+1}$ is an increasing sequence in $\mathcal{O}_c(G)$ such that
    $$\langle K_i^{15}\cap G|_{V_i}\rangle\subseteq K_{i+1} \ \forall 0\leq i\leq n.$$
    Then $\langle K_0\cap G|_V\rangle\subseteq K_{n+1}^5$.
\end{lemma}
\begin{proof}
    This follows inductively from Lemma \ref{lem:gluing}. 
\end{proof}

As a simple application of this Lemma, we obtain:

\begin{lemma}\label{lem:gluing1}
    Let $G$ be an \'etale groupoid.
	Suppose $G^0=X_0\sqcup X_1\sqcup\cdots\sqcup X_{n-1}$ is a clopen partition of $G^0$. Assume that there exists an increasing sequence $K_0\subseteq K_1\subseteq\cdots\subseteq K_{n}$ in $\mathcal{O}_c(G)$ such that
 each restriction $G|_{X_i}$ has $(K_i^{15}\cap G|_{X_i},K_{i+1})-\mathrm{dad}$ at most $d$. Then  $G$ has $(K_0,K_n^5)$-dynamic asymptotic dimension at most $d$.
\end{lemma}

The next Lemma is well-known:
\begin{lemma}
	Let $G$ be a compact principal ample groupoid. Then there exists a clopen fundamental domain $Y_*\subseteq H^0$, i.e. $Y_*$ meets each $H$-orbit in $H^0$ exactly once.
\end{lemma}

The key technicalities for the main result of this section are contained in the following Lemma. The main obstacle in generalising the proof presented in \cite{CJMSTD20} for group actions to the case of general groupoids is that the decompositions of the fibres $G^x$ obtained from the finite asymptotic dimension assumption are not uniform as $x$ varies. In contrast, if we take a decomposition of $\Gamma=T_0\sqcup \ldots\sqcup T_d$ as in the definition of asymptotic dimension at most $d$, then $T_i\times X$ gives a decomposition of the transformation groupoid, which works uniformly over $X$.
\begin{lemma}\label{lem:asdim->dad} Let $G$ be a pricipal, ample groupoid and let $Y\subseteq G^0$ be a clopen subset. Assume further, that for given $K, L\subseteq G$ compact open, the groupoid
	$$H_K(Y)\coloneqq \langle K\cap G|_Y\rangle$$
	is compact.
	If $(K,L)-\mathrm{asdim}(G)\leq d$, then $(K,L)-\mathrm{dad}(G|_Y)\leq d$.

\end{lemma}
\begin{proof}
		Fix $x\in Y$ for the moment. We can use the assumption to obtain a partition $G^x=\widetilde{T}_0^x\sqcup\ldots\sqcup \widetilde{T}_d^x$ where each $\widetilde{T}^x_i$ further decomposes as $$\widetilde{T}^x_i=\bigsqcup_{j}\widetilde{D}^x_{i,j}$$ such that each $\widetilde{D}^x_{i,j}$ is $L$-bounded and $\widetilde{D}^x_{i,j_1}$ and $\widetilde{D}^x_{i,j_2}$ are $K$-disjoint for all $j_1\neq j_2$. Intersecting each component of this partition with $H_K(Y)$ gives a partition of $H_K(Y)^x$ with the same properties. Let $D_{i,j}^x\coloneqq\widetilde{D}_{i,j}^x\cap H_K(Y)$ and $T_i^x\coloneqq\widetilde{T}_{i,j}^x\cap H_K(Y)$ denote these intersections.
		
        Our first goal is to show that we can make this decomposition work uniformly over a neighbourhood $V_x$ of $x$.
		Note that since $H_K(Y)$ is compact, only finitely many of the sets $D_{i,j}^x$ are non-empty and each of them has to be finite. For fixed $i,j$ write $D_{i,j}^x=\{g_1,\ldots,g_n\}$. Using continuity of the product map we can find a compact open bisection $U_k$ around each $g_k$ such that $U_k^{-1}U_l\subseteq L$. We may also assume that the $U_k$ are pairwise disjoint and all have the same range (if that's not the case, replace $U_k$ by $U_k\cap r^{-1}(\bigcap_l r(U_l))$). Let $D_{x,i,j}\coloneqq\bigcup_{l=1}^n U_l$ (we put the $x$ in the subscript to indicate that this set still depends on $x$, but need not be a subset of $G^x$). By construction $D_{x,i,j}$ is compact open and $L$-bounded. By enumerating the set $\{j\mid D_{i,j}^x\neq \emptyset\}$ we can successively shrink the sets $D_{x,i,j}$ to make sure that they remain $K$-disjoint. Let  $V_x\coloneqq\bigcap_{i,j} r(D_{x,i,j})$. Then $V_x$ is a clopen neighbourhood of $x$. Replace each $D_{x,i,j}$ by $D_{x,i,j}\cap r^{-1}(V_x)$ and let $T_{x,i}\coloneqq\bigsqcup_i D_{x,i,j}$. Then for each $y\in V_x$ $H_K(Y)^y=T_{x,0}^y\sqcup\ldots\sqcup T_{x,d}^y$ and each $T_{x,i}^y$ further decomposes as
		$$T^y_{x,i}=\bigsqcup_{j}D^y_{x,i,j}.$$
        For $y=x$ we recover the original decomposition found in the beginning, i.e. $D^x_{x,i,j}=D_{i,j}^x$.
        
        Using that $Y$ is compact we can thus partition $Y$ by finitely many sets $V_1,\ldots V_p$ such that for each $k\in \{1,\ldots, p\}$ there exists a partition $H_K(Y)\cap G^{V_k}=T_{k,0}\sqcup\ldots T_{k,d}$ such that each partition further decomposes as a $K$-disjoint union $$T_{k,i}=\bigsqcup_{j}D_{k,i,j},$$
        of $L$-bounded sets.
        
		Let $D_{i,j}=\bigsqcup_{k=1}^p D_{k,i,j}$ and $T_i=\bigsqcup_{k=1}^p T_{k,i}$. Pick a clopen subset $Y_*\subseteq Y$ that meets each $H_K(Y)$-orbit exactly once and set
		$$U_i\coloneqq\{s(g)\mid g\in H_K(Y)\cap T_i, r(g)\in Y_*, s(g)\in Y\}.$$ Then $U_0,\ldots, U_d$ are clearly clopen subsets of $Y$.
		We have to show that $\langle K\cap G|_{U_i}\rangle$ is contained in $L$. To this end write an arbitrary $g\in \langle K\cap G|_{U_i}\rangle$ as $g=g_1g_2\cdots g_n$ with $g_k\in K$ and $s(g_k),r(g_k)\in U_i$. It follows that for each $1\leq k\leq n$ there exist $h_k,h'_k\in T_i$ with $r(h_k),r(h_k')\in Y_*$ and $s(g_k)=s(h_k)$ and $r(g_k)=s(h'_k)$. Note that the set $\{r(h_k),r(h'_k)\mid 1\leq k\leq n\}$ is contained in a single $H_K(Y)$-orbit. Since $Y_*$ meets each $H_K(Y)$-orbit exactly once we must have $r(h_1)=r(h'_1)=\ldots=r(h_k)=r(h'_k)$. Hence $h_k^{-1}h'_kg_k\in \mathrm{Iso}(G)=G^0$, so using principality of $G$ we get $g_k=(h'_k)^{-1}h_k$. Moreover, since $g_k\in K$, the elements $h_k$ and $h_k'$ are in the same $D_{i,j}$. So we can write 
		$$g=(h'_1)^{-1}h_1(h'_2)^{-1}h_2\cdots (h'_n)^{-1}h_n.$$
		Note further, that since $s(h_k(h'_{k+1})^{-1})=r(h'_{k+1})=r(h_k)=r(h_k(h'_{k+1})^{-1})$ we can use principality again to conclude that $h_k=h'_{k+1}$. In particular, there exists a unique $j$ such that $h_k,h'_k\in D_{i,j}$ for all $1\leq k\leq n$. Putting these two facts together we obtain
		$$g=(h'_1)^{-1}h_n\in D_{i,j}^{-1}D_{i,j}\subseteq L$$
		as desired.
\end{proof}
We can now proceed with the proof of the second half of Theorem \ref{Thm:dad=sup asdim}.
\begin{satz}
	Let $G$ be a principal \'etale groupoid with $\sigma$-compact and totally disconnected unit space. If $\mathrm{dad}(G)<\infty$, then
	$$\mathrm{dad}(G)\leq \mathrm{asdim}(G,\mathcal{E}_G).$$
	If $G^0$ is compact, equality holds.
\end{satz}
\begin{proof}
    We first prove the result under the additional assumption that $G^0$ is compact.
	Let $\mathrm{dad}(G)\leq D$ and $d=\mathrm{asdim}(G)$. Let $K$ be a compact open subset of $G$. We want to show $\mathrm{dad}(G)\leq d$. Inductively find an increasing sequence $(K_i)_i$ of compact opens such that $K\cup G^0\subseteq K_0$ and such that $G$ has $(K_i^{15}, K_{i+1})$-asdim at most $d$ for every $i$. Now use the assumption $\mathrm{dad}(G)\leq D$ to find clopen subsets $X_0,\ldots X_D$ such that 
	$$H_i\coloneqq \langle K_D\cap G|_{X_i}\rangle$$
	is a compact open subgroupoid. Apply Lemma \ref{lem:asdim->dad} to each $X_i$ to see that $(K_i^{15},K_{i+1})-\mathrm{dad}(G|_{X_i})\leq d$. Hence Lemma \ref{lem:gluing} implies that $(K_0,K_D^5)-\mathrm{dad}(G)\leq d$. As we started with an arbitrary $K$, this implies that $\mathrm{dad}(G)\leq d$ as desired.
	
	If $G^0$ is just locally compact and $\sigma$-compact, then there exists a nested sequence of compact open sets $W_n\subseteq G^0$ covering $G^0$. Let $G_n\coloneqq G|_{W_n}$ be the restriction of $G$ to $W_n$. Then $(G_n)_n$ is a nested sequence of compact open subgroupoids of $G$ and hence $\mathrm{dad}(G_n)\leq \mathrm{dad}(G)<\infty$. By the first part of this proof, $\mathrm{dad}(G_n)\leq \mathrm{asdim}(G_n,\mathcal{E}_{G_n})\leq \mathrm{asdim}(G,\mathcal{E}_G)$ and hence the result follows from Proposition \ref{prop:limit}.
\end{proof}

\begin{bem}
	Let us remark that the assumption $\mathrm{dad}(G)<\infty$ in the previous theorem cannot be dropped:
	Consider for example the free group $\mathbb{F}_2$. Since $\mathbb{F}_2$ is residually finite, it admits a decreasing sequence $N_1\supseteq N_2\supseteq \ldots$ of finite index normal subgroups such that $\bigcap_{k\in \mathbb N} N_k=\{e\}$. For each $k\in \mathbb N$ there is a canonical surjective group homomorphism $\mathbb{F}_2/N_{k+1}\to \mathbb{F}_2/N_k$. Let $X$ be the inverse limit of the sequence $\mathbb{F}_2/N_1 \leftarrow \mathbb{F}_2/N_2 \leftarrow\cdots$, which as a topological space is a Cantor set. The group $\mathbb{F}_2$ acts from the left on each quotient $\mathbb{F}_2/N_k$ and this induces a free and minimal action on $X$. The space $X$ also admits a unique Borel probability measure $\mu$, induced by the uniform probability measures on the (finite) quotients $\mathbb{F}_2/N_k$.
    It follows that the transformation groupoid $\mathbb{F}_2\ltimes X$ is non-amenable, and hence in particular it cannot have finite dynamic asymptotic dimension by \cite[Corollary~8.25]{GWY17}.
    On the other hand we have $\mathrm{asdim}(\mathbb{F}_2 \ltimes X)=\mathrm{asdim}(\mathbb{F}_2)=1$.
\end{bem}

\bibliographystyle{plain}
\bibliography{references}
\end{document}